\documentclass[11pt]{amsart}
\usepackage{lipsum}
\usepackage{amsmath, amsthm, amssymb}
\usepackage{natbib}
\usepackage{lineno,hyperref}
\usepackage[T1]{fontenc}
\usepackage[utf8]{inputenc}
\usepackage{authblk}
\makeatletter
\g@addto@macro{\endabstract}{\@setabstract}
\newcommand{\authorfootnotes}{\renewcommand\thefootnote{\@fnsymbol\c@footnote}}%
\makeatother

\newtheorem{prop}{Proposition}[section]
\newtheorem{lemma}[prop]{Lemma}

\newtheorem{theorem}[prop]{Theorem}
\newtheorem{defn}[prop]{Definition}

\newtheorem{corollary}[prop]{Corollary}
\newtheorem{proposition}[prop]{Proposition}
\newtheorem{remark}[prop]{Remark}

\newcommand{\BX}{{\bf X}}
\newcommand{\BY}{{\bf Y}}

\begin{document}

\title[Location of Supremum of sssi Processes]{Location of the Path Supremum for Self-similar Processes with Stationary Increments}
\thanks{Email: yi.shen@uwaterloo.ca}
\renewcommand\Authands{ and }

\thanks{}

\vspace{.5ex}

\maketitle
\begin{center}

  \normalsize
  \authorfootnotes
  Yi Shen\par \bigskip
  Department of Statistics and Actuarial Science, University of Waterloo. Waterloo, ON N2L 3G1, Canada. \par\bigskip
\thanks{Email: yi.shen@uwaterloo.ca}
\thanks{This work is supported by NSERC grant.}

\end{center}

\begin{abstract}
In this paper we consider the distribution of the location of the path supremum in a fixed interval for self-similar processes with stationary increments. To this end, a point process is constructed and its relation to the distribution of the location of the path supremum is studied. Using this framework, we show that the distribution has a spectral-type representation, in the sense that it is always a mixture of a special group of absolutely continuous distributions, plus point masses on the two boundaries. Bounds on the value and the derivatives of the density function are established. We further discuss self-similar L\'{e}vy processes as an example. Most of the results in this paper can be generalized to a group of random locations, including the location of the largest jump, \textit{etc}.
\end{abstract}


\section{Introduction}
Self-similar processes are stochastic processes whose distributions do not change under proper rescaling in time and space. The study of self-similar processes as a unified concept dates back to \cite{lamperti:1962}, and this class of processes have attracted attention of researchers from various fields since then, due to their theoretical tractability and broad applications. The book \cite{embrechts:maejima:2002}, the lecture note \cite{chaumont:2010}, and the review papers \cite{embrechts:maejima:2000} and \cite{pardo:2007} are all excellent sources for general introduction and existing results. A special subclass of self-similar processes, self-similar processes with stationary increments, or \textit{ss,si} processes in short, are of particular interest. They combine the two probabilistic symmetries given by the self-similarity and the stationary increments, and include famous examples such as fractional Brownian motions and self-similar L\'{e}vy processes.

In this paper, we consider the distributional properties of the location of the path supremum over a fixed interval for self-similar processes with stationary increments. Compared to the values of the extremes, their locations received relatively less attention. On one hand, there exist results for some special cases. For instance, the distribution of the location of path supremum for a Brownian motion is well-known as the (third) arcsin law. More generally, the result for self-similar L\'{e}vy processes was given in \cite{chaumont:2013}. While the exact result for fractional Brownian motions remains unclear, approximate distributions were studied in \cite{delorme:wiese:2015} using perturbation theory. On the other hand, there has not been any structural study of the distribution of the location of path supremum for general self-similar processes. Our goal in this paper is to establish a framework which works for general self-similar processes, and to derive properties for the distribution of the location of path supremum.

\section{Basic Settings}
Let ${\mathbf X}=\{X(t)\}_{t\in\mathbb R}$ be a stochastic process defined on some probability space $(\Omega, \mathcal F, P)$ and taking values in $\mathbb R$, whose path is almost surely c\`{a}dl\`{a}g. $\{X(t)\}_{t\in\mathbb R}$ is said to be self-similar, if $\{X(at)\}$ and $\{a^HX(t)\}$ have the same distribution, for some $H\geq 0$. The constant $H$ is called the exponent of the self-similar process. In this work, $\{X(t)\}_{t\in \mathbb R}$ is always assumed self-similar. We further assume that $\{X(t)\}_{t\in\mathbb R}$ has stationary increments. Such a self-similar process with stationary increments is often referred as a $H-ss,si$ process, or a $ss,si$ process when it is not necessary to specify $H$.

We are interested in the distribution of the location of the path supremum for the $ss,si$ process $\{X(t)\}_{t\in\mathbb R}$ over an interval with a fixed length, say, $T>0$. By the stationarity of the increments, the distribution will be the same for any interval with length $T$, and consequently, the interval can be chosen as $[0,T]$. In this case, we denote the location of the path supremum of $\{X(t)\}_{t\in\mathbb R}$ over $[0,T]$ by
$$
\tau_{{\mathbf X},T}:=\inf\left\{t\in[0,T]: \limsup_{s\to t} X(s)=\sup_{s\in[0,T]}X(s)\right\},
$$
where the infimum means that in case where the supremum is achieved at multiple locations, the leftmost point among them is taken.

Alternatively, one can first define
$$
\tau_{{\mathbf X},T}:=\inf\left\{t\in[0,T]: X(t)=\sup_{s\in[0,T]}X(s)\right\}
$$
for stochastic processes with upper semi-continuous paths, in which case the supremum can be indeed achieved at some point. Then for a c\`{a}dl\`{a}g process $\BX$, define its location of path supremum over $[0,T]$ as that of the modified process
$$
X'(t)=X(t-)\vee X(t),
$$
which is now upper semicontinuous. Its path have the property that the left limit and right limit $X'(t-)$ and $X'(t+)$ exists for all $t\in\mathbb R$, and $X'(t)=X'(t-)\vee X'(t+)$. Denote by $D'$ the collection of all such functions on $\mathbb R$. Note that if $\BX$ is $ss,si$, so is $\BX'$. Therefore, in the rest of the paper, we can always assume that the process $X$ has paths in $D'$, thus in particular, is upper semicontinuous.

It is not difficult to check that $\tau_{\BX, T}$ is a well-defined random variable. Moreover, since the process is self-similar, the distributions of $\tau_{\BX, T}$ and $\tau_{\BX, 1}$ are the same up to a scaling. Therefore we focus on the case where $T=1$, and use the simplified notation $\tau_{\BX}=\tau_{\BX, 1}$.

Additionally, define set
$$
M=\{\omega\in\Omega: X(t_i)=\sup_{s\in[0,1]}X(s) \text{ for at least two different } t_1, t_2\in [0,1]\},
$$
and assume that the process $\{X(t)\}_{t\in\mathbb R}$ satisfies\\

\textbf{Assumption U.} $P(M)=0$.\\

Most processes that we encounter do satisfy Assumption U. A necessary and sufficient condition for this assumption can be found in \cite{pimentel:2014}. Under Assumption U, the supremum in $[0,1]$ is attained at a unique point, and the infimum in the definition of $\tau_{\BX}$ can be removed. Notice that Assumption U also excludes the case where $H=0$ provided that the process is stochastically continuous at $t=0$, since in that case $X(t)$ must be a constant, which is trivial anyway. Hence we always assume $H>0$ for the rest of the paper.

In \cite{shen:2016}, the author showed that for a stochastic process $\{X(t)\}_{t\in\mathbb R}$ with stationary increments, the distribution of $\tau_{\BX}$ can have point masses on the two boundaries 0 and 1, but must be absolutely continuous in the interior of the interval $[0,1]$, and its density function, denoted as $f(t)$, can be taken as the right derivative of the cumulated distribution function of $\tau_{\BX}$. Moreover, this version of the density function is c\`adl\`ag everywhere.

In the presence of self-similarity, it turns out that the distribution of $\tau_{\BX}$ is closely related to a point process, constructed as below.

For $t\in \mathbb R$, define $l(t)=\inf\{s>0:X(t-s)\geq X(t)\}$ and $r(t)=\inf\{s>0: X(t+s)\geq X(t)\}$, with the tradition that $\inf(\phi)=+\infty$. Intuitively, $l(t)$ and $r(t)$ are the distances by which the process will return to the level $X(t)$ or higher at the left and the right of the point $t$. It is clear that $t$ is a (strict) local maximum if and only if both $l(t)$ and $r(t)$ are strictly positive.

Let $S=\{s\in{\mathbb R}: l(s)>0, r(s)>0\}$ be the set of all the local maxima of $\BX$. Notice that since $|s_1-s_2|\geq \min\{l(s_1), r(s_1), l(s_2), r(s_2)\}$ for any $s_1,s_2\in S$, $S$ is at most countable. For each $s_i\in S\cap[0,1]$, define point $\epsilon_i=(l(s_i), r(s_i))$, then $\epsilon_i$ is a point in the first quadrant of $\overline{\mathbb R}^2$, where $\overline{\mathbb R}=[-\infty,
\infty]$. The collection of these points, denoted as $\mathcal E$, or more precisely, the (random) counting measure determined by it, denoted as $\xi:=\sum_{\epsilon_i\in \mathcal E}\delta_{\epsilon_i}$, forms a point process in $({\overline{\mathbb R}^+})^2$, where $\overline{\mathbb R}^+=(0,\infty]$.

Let $\nu$ be the mean measure of the point process $\xi$:
\begin{equation}\label{e:measure}
\nu(A)=E(\xi(A)) \quad \text{ for every } A\in \mathcal B((\overline{\mathbb R}^{+})^2),
\end{equation}
where $\mathcal B((\overline{\mathbb R}^{+})^2)$ is the Borel $\sigma-$algebra on $(\overline{\mathbb R}^+)^2$, with $+\infty$ treated as a separate point in $\overline{\mathbb R}^+$. Again, since the points in $\mathcal E$ have the property that $|s_1-s_2|\geq \min\{l(s_1), r(s_1), l(s_2), r(s_2)\}$, $\nu(A)$ is finite whenever the set $A$ is bounded away from the axes.

\section{Supremum location of ss,si processes}\label{s:general}

We start by exploring the structure of the measure $\nu$. Firstly, the following result shows that $\nu$ has mass 0 on the boundaries at $+\infty$. As a result, we can effectively remove infinity from the definition of $\xi$ when only $\nu$ is considered.

\begin{proposition}
Let $\BX$ be a $ss,si$ process satisfying Assumption U, and $\nu$ be defined as at the end of Section 2.. Then
$$\nu({\mathbb R}^+\times\{+\infty\})=\nu(\{+\infty\}\times {\mathbb R}^+)=\nu(\{+\infty\}\times \{+\infty\})=0.$$
\end{proposition}

\begin{proof}
First, notice that the set $S_\infty:=\{t\in {\mathbb R}: l(t)=r(t)=+\infty\}$ contains at most one single point, which is the location of the strict global maximum of the process $\BX$ over the whole real line. Thus by the stationarity of the increments, $P(S_\infty\cap[0,1]\neq\phi)=\nu(\{+\infty\}\times\{+\infty\})=0$.

For the rest part of the proposition and future use, we introduce the notion of ``compatible set''. Denote by $(D',\mathcal C)$ the collection of all modified c\`{a}dl\`{a}g paths equipped with the cylindrical $\sigma-$field, and $(M_P, {\mathcal M}_P)$ the standard measurable space for point processes on the real line.

\begin{defn}\label{d:compatible}
A compatible set $I$ is a measurable mapping from $(D',\mathcal C)$ to $(M_P, {\mathcal M}_P)$, satisfying
\begin{enumerate}
\item (Shift compatibility) $I(\theta_c\circ (g+d))=\theta_c\circ I(g)$  for all $g\in D'$ and $c,d\in\mathbb R$, where $\theta_c$ is the shift operator: $\theta_c\circ g(t):=g(t+c), t\in\mathbb R$; $\theta_c\circ \Gamma:=\Gamma-c, \Gamma\in M_P$, and $(g+d)(t):=g(t)+d$.
\item (Scaling compatibility) $I(d\cdot g(c ~\cdot))=c^{-1}I(g(\cdot))$ for all $g\in D'$ and $c,d\in\mathbb R^+$.
\end{enumerate}
\end{defn}

By the above definition and the corresponding probabilistic symmetries possessed by $ss,si$ processes, it is clear that the distribution of the point process $I(\BX)$ will be both stationary and scaling invariant if the underlying process $\BX$ is $ss,si$ and $I$ is a compatible set. Then we have
\begin{lemma}\label{l:compatible}
Let $\BX$ be a $ss,si$ process with paths in $D'$, and $I$ be a compatible set. Then
$P(I(\BX) \text{ is dense in } {\mathbb R})+P(I(\BX)=\phi)=1$.
\end{lemma}
That is, a compatible set of a $ss,si$ process is either empty or dense. The proof of this lemma is very simple.
\begin{proof}
Assume that $P(I(\BX) \text{ is neither dense nor empty })>0$. By stationarity of $I(\BX)$, this implies that
$$
P(I(\BX)\neq \phi, d(0,I(\BX))>0)>0,
$$
where $d$ denotes the Euclidean distance between points or sets. However, since the distribution of $I(\BX)$ is invariant under rescaling, so is the distribution of $d(0, I(\BX))$, which implies that $d(0, I(\BX))\in\{0,+\infty\}$ almost surely, contradicting with the above result. Thus we conclude that
$$
P(I(\BX) \text{ is dense in } {\mathbb R})+P(I(\BX)=\phi)=1.
$$
\end{proof}

Now consider the set $S_r:=\{s\in S: r(s)=+\infty\}$. Note that $S_r$ is a compatible set. Therefore, by Lemma \ref{l:compatible}, $S_r$ is either dense or empty almost surely. Assume $S_r$ is dense. Then for any fixed $t>0$ and $s\in S_r$ such that $s<t$, $X(s)>X(t)$. Taking a sequence of such $s$ converging to $0$ leads to the result that $X(0)\geq X(t)$ almost surely. Hence by stationarity of increments, $X(t_1)\geq X(t_2)$ almost surely for any $t_1<t_2$, which implies that the process has monotonic paths. However, the definition of $S$ requires that both $l(s)$ and $r(s)$ be strictly positive for $s\in S$, thus $S$ should be empty if the path is monotonic, contradicting with the assumption that $S_r$ is dense. Therefore we conclude that $S_r$ is empty almost surely.

Symmetrically, $S_l:=\{s\in S: l(s)=+\infty\}$ is empty almost surely. As a result, $\nu({\mathbb R}^+\times\{+\infty\})=\nu(\{+\infty\}\times {\mathbb R}^+)=0$.

\end{proof}

Note that if $s$ is the location of a local maximum of $\BX$, then for any $a>0$, $s'=as$ is the location of a local maximum of $\BY$ defined by $Y(at)=X(t), t\in \mathbb R$. Moreover, $l(s')=al(s)$ and $r(s')=ar(s)$. Therefore by self-similarity, for any $A\in {\mathcal B}$, we have
\begin{equation}\label{e:trans}
\nu(a(A))=a^{-1}\nu(A),
\end{equation}
where $aA:=\{(al,ar): (l,r)\in A\}$, and the factor $a^{-1}$ on the right hand side comes from the fact that the measure $\nu$ still counts the expected number of qualified points in an interval with length 1 rather than length $a$.

Define bijection $\Psi: (l,r)\to (u,v)$ by
$$
u:=l, v:=\frac{l}{l+r},
$$
then the relation (\ref{e:trans}) becomes
$$
\nu'(\varphi^u_a(A'))=a^{-1}\nu'(A'), \quad A'\in {\mathcal B},
$$
where $\varphi^u_a(A')=\{(au,v): (u,v)\in A'\}$, and $\nu'=\nu\circ\Psi^{-1}$. This observation immediately leads to a factorization

\begin{equation}
\nu'=\eta\times\mu,
\end{equation}
where $\mu$ is a measure on $\left((0,1), {\mathcal B}(0,1)\right)$, $\eta$ is a measure on $\left({\mathbb R}^+, {\mathcal B}({\mathbb R}^+)\right)$ satisfying
$$
\eta(aB)=a^{-1}\eta(B)
$$
for any $a>0, B\in{\mathcal B}({\mathbb R}^+)$. Therefore we have
\begin{lemma}
$\eta$ is an absolutely continuous measure with density
$$
g(u)=cu^{-2}, \quad u>0
$$
for some positive constant $c$.
\end{lemma}

The following theorem reveals a key relation between $f(t)$ and $\nu$.

\begin{theorem}\label{t:frame}
Let $\{X(t)\}_{t\in\mathbb R}$ be $H-ss,si$ for $H>0$, satisfying Assumption U. $f(t)$ is the density in $(0,1)$ of $\tau_\BX$, and $\nu$ is defined as in (\ref{e:measure}). Then
$$
f(t)=\nu([t,\infty)\times[1-t,\infty)),\quad t\in(0,1).
$$
\end{theorem}

\begin{proof}
Let $0<\varepsilon <\min\{t,\frac{1-t}{2}\}$. Notice that
\begin{align*}
&\{\text{there exists }s\in S\cap [t,t+\varepsilon] \text{ satisfying } l(s)>t+\varepsilon, r(s)> 1-t\}\\
\subseteq & \{\tau_{\BX}\in[t,t+\varepsilon]\}\\
\subseteq & \{\text{there exists }s\in S\cap [t,t+\varepsilon] \text{ satisfying } l(s)\geq t, r(s)\geq 1-t-\varepsilon\}.
\end{align*}
Since for any point $s$ in the set on the left or the right side, $\min\{l(s), r(s)\}\geq\min\{t,1-t-\varepsilon\}>\varepsilon$, there can only be at most one such point in the interval $t,t+\varepsilon$. Thus by the stationarity of the increments,
\begin{align*}
& P(\{\text{there exists }s\in S\cap [t,t+\varepsilon] \text{ satisfying } l(s)>t+\varepsilon, r(s)> 1-t\})\\
=& E(|s_i\in S\cap[t,t+\varepsilon]: \epsilon_i\in(t+\varepsilon,\infty)\times(1-t, \infty)|)\\
=& \varepsilon E(|s_i\in S\cap[0,1]: \epsilon_i\in(t+\varepsilon,\infty)\times(1-t, \infty)|)\\
=& \varepsilon E\left(\sum_{\epsilon_i\in{\mathcal E}}\delta_{\epsilon_i}((t+\varepsilon,\infty)\times(1-t, \infty))\right)\\
=& \varepsilon \nu((t+\varepsilon,\infty)\times(1-t, \infty)),
\end{align*}
where $|\cdot|$ for a set gives the number of the elements in the set. Similarly,
\begin{align*}
&P(\{\text{there exists }s\in S\cap [t,t+\varepsilon] \text{ satisfying } l(s)\geq t, r(s)\geq 1-t-\varepsilon\})\\
=&\varepsilon \nu([t,\infty)\times[1-t-\varepsilon, \infty)).
\end{align*}
Thus
\begin{align*}
&\varepsilon \nu((t+\varepsilon,\infty)\times(1-t, \infty))\\
\leq &P(\tau_{\BX}\in[t,t+\varepsilon])\\
\leq &\varepsilon \nu([t,\infty)\times[1-t-\varepsilon, \infty)).
\end{align*}
Recall that $f(t)$ is right continuous with left limits, and equals to the right derivative of the cumulative distribution function of $\tau_{\BX}$ everywhere. Therefore by dividing all the expressions by $\varepsilon$ and taking limit $\varepsilon\to 0$, we have
$$
\lim_{\varepsilon\to 0}\nu((t+\varepsilon,\infty)\times(1-t, \infty))\leq f(t) \leq \lim_{\varepsilon\to 0}\nu([t,\infty)\times[1-t-\varepsilon, \infty)).
$$
It is not difficult to see that
$$
\Psi([t,\infty)\times[1-t,\infty))=\{(u,v):u\geq h(v,t)\},
$$
where
$$
h(v,t)=\begin{cases}
t & \quad 0<v<t,\\
\frac{v}{1-v}(1-t) & \quad t\leq v<1.
\end{cases}
$$
Hence
\begin{equation}\label{e:integral}
\begin{aligned}
&\nu([t,\infty)\times[1-t,\infty))\\
=&\nu'(\{(u,v):u\geq h(v,t)\})\\
=&\int_{v=0}^1\eta([h(v,t),\infty)) \mu(dv)\\
=&\int_{v=0}^1\left(\int_{u=h(v,t)}^\infty c u^{-2}du\right) \mu(dv).
\end{aligned}
\end{equation}
As a result, the boundary of the set $[t,\infty)\times[1-t,\infty)$ is a null set under $\nu$. Therefore
$$
\lim_{\varepsilon\to 0}\nu((t+\varepsilon,\infty)\times(1-t, \infty))=\lim_{\varepsilon\to 0}\nu([t,\infty)\times[1-t-\varepsilon, \infty))=\nu([t,\infty)\times[1-t,\infty)).
$$
\end{proof}

It is now straightforward to derive the following spectral-type result using relation (\ref{e:integral}).

\begin{theorem}\label{t:spectral}
Let $\{X(t)\}_{t\in\mathbb R}$ and $f(t)$ be defined as in Theorem \ref{t:frame}. Then
\begin{equation}\label{e:spectral}
f(t)=\int_0^1 f_v(t)\mu_1(dv), \quad 0<t<1,
\end{equation}
where
$$
f_v(t)=\begin{cases}
\frac{1-v}{-v\ln(v)-(1-v)\ln(1-v)}(1-t)^{-1} & \quad t\leq v,\\
\frac{v}{-v\ln(v)-(1-v)\ln(1-v)}t^{-1} & \quad t>v,\\
\end{cases}
$$
and $\mu_1$ is a sub-probability measure on $(0,1)$ (\textit{i.e.,} $\mu_1(0,1)\leq 1$).\\
\end{theorem}

\begin{proof}
The proof of Theorem \ref{t:spectral} is a simple rewriting of (\ref{e:integral}). More precisely, we have
$$
\begin{aligned}
f(t)= &\nu([t,\infty)\times[1-t,\infty))\\
=&\int_{v=0}^1\left(\int_{u=h(v,t)}^\infty c u^{-2}du\right) \mu(dv)\\
=&\int_{v=0}^1 ch^{-1}(v,t)\mu(dv)\\
=&\int_{v=0}^1f_v(t)c(v)\mu(dv),
\end{aligned}
$$
where
$$
c(v)=\frac{c(-v\ln(v)-(1-v)\ln(1-v))}{v}.
$$
Defining measure $\mu_1$ by $\frac{d\mu_1}{d\mu}(v)=c(v)$ leads to the desired expression. Finally, since $\int_{0}^1f_v(t)dt=1$ for any $v\in(0,1)$ and $\int_0^1f(t)dt\leq 1$, $\mu_1$ is a sub-probability measure. Notice that $\mu$ is not necessary a probability measure due to the potential mass of the distribution of the path supremum on the boundaries 0 and 1.
\end{proof}

The next result gives a universal entropy-type upper bound for the density function $f(t)$. It can be obtained using only the basic properties of self-similarity, but here we are going to prove it using the result of Theorem \ref{t:spectral}.

\begin{corollary}\label{p:bounds}
For any given $t\in(0,1)$,
\begin{equation}\label{e:upperbound}
f(t)\leq (-t\ln(t)-(1-t)\ln(1-t))^{-1}.
\end{equation}
\end{corollary}

\begin{proof}
By Theorem \ref{t:spectral}, it suffices to check (\ref{e:upperbound}) for $f_v(t)$ for all $0<v<1$, which can be done using fundamental calculus.
\end{proof}

 Corollary\ref{p:bounds} is a significant improvement of the corresponding result derived in \cite{shen:2016} for general processes with stationary increments but not necessarily with self-similarity. More precisely, the upper bound of $f(t)$ is improved from $\max\{t^{-1},(1-t)^{-1}\}$ to the current form. The factor of improvement varies from $-\ln(t)$ when $t\to 0$ and $-\ln(1-t)$ when $t\to 1$ to $\frac{2}{\ln(2)}$ when $t=\frac{1}{2}$.

\begin{remark}
In the excellent work of \cite{molchan:khokhlov:2004} the authors proved that
$$
f(t)\leq f(s)\max\left(\frac{s}{t},\frac{1-s}{1-t}\right)
$$
for any $s,t\in(0,1)$. In particular, $f$ is always continuous. Moreover, assuming the existence of the left and the right derivatives, denoted as $f'(t-)$ and $f'(t+)$ respectively, the above result easily leads to the following bounds:
\begin{equation}\label{e:derivative}
f'(t-)\leq \frac{f(t)}{1-t},
\end{equation}
\begin{equation}
f'(t+)\geq -\frac{f(t)}{t}.
\end{equation}

Our framework provides an alternative way to derive these bounds: similar as in Corollary \ref{p:bounds}, one can directly check that the above bounds are satisfied by all the basis functions $f_v(t), 0<v<1$, hence they must hold for all density functions $f$. This method also guarantees the existence of the left and the right derivatives.

\end{remark}

The following immediate corollary of Theorem \ref{t:spectral} gives bounds for the expectation of any function of the location of the path supremum. The proof is omitted.

\begin{corollary}\label{c:bounds}
Let $\{X(t)\}_{t\in\mathbb R}$ be $H$-ss,si for $H>0$, satisfying Assumption U, and $\tau_{\BX}$ be the location of its path supremum over $[0,1]$. Let $g$ be a bounded, or non-negative, measurable function on $[0,1]$. Then
$$
\begin{aligned}
&\min\left\{g(0), g(1), \inf_{v\in(0,1)}\int_0^1g(t)f_v(t)dt\right\}\\
\leq & {\mathbb E}(g(\tau_{\BX}))\\
\leq & \max\left\{g(0), g(1), \sup_{v\in(0,1)}\int_0^1g(t)f_v(t)dt\right\}.
\end{aligned}
$$
\end{corollary}

Corollary \ref{c:bounds} can be used to derive, for example, the upper bound for the probability that the path supremum falls into an interval $[c,d]$: $P(\tau_{\BX}\in[c,d])$.

In many cases the process $\BX$ is time-reversible, \textit{i.e.}, $\{X_t\}_{t\in \mathbb R}\stackrel{d}{=}\{X(-t)\}_{t\in\mathbb R}$. For instance, all fractional Brownian motions are time-reversible. This property further improves the spectral-type representation result and the related bound.

\begin{proposition}
Let $\{X(t)\}_{t\in \mathbb R}$ be a time-reversible, $ss,si$ process, and $f(t)$ be the density in $(0,1)$ of the location of the path supremum for $\{X(t)\}_{t\in\mathbb R}$. Then
$$
f(t)=\int_0^{1/2}\tilde{f}_v(t)\tilde{\mu}_1(dv),
$$
where
$$
\tilde{f}_v(t)=\left\{
\begin{aligned}
& \frac{1}{2(-v\ln(v)-(1-v)\ln(1-v))}(1-t)^{-1} & & 0<t<v\\
& \frac{v}{2(-v\ln(v)-(1-v)\ln(1-v))}(t^{-1}+(1-t)^{-1}) &  & v\leq t < 1-v\\
& \frac{1}{2(-v\ln(v)-(1-v)\ln(1-v))}t^{-1} &  &1-v\leq t <1
\end{aligned}
\right.,
$$
and $\tilde{\mu}_1$ is a sub-probability measure on $\left(0,\frac{1}{2}\right]$.
\end{proposition}

To see this result, simply use the fact that $\mu_1$ in Theorem \ref{t:spectral} now needs to be symmetric due to the time-reversibility, then define $\tilde{f}_v(t)=\frac{1}{2}(f_v(t)+f_{1-v}(t))$. We omit the details. The corresponding upper bound for $f(t)$ becomes
$$
f(t)\leq\left\{
\begin{aligned}
&\frac{1}{2(1-t)}(-t(\ln(t))-(1-t)\ln(1-t))^{-1} & & 0<t<\frac{1}{2},\\
&\frac{1}{2t}(-t(\ln(t))-(1-t)\ln(1-t))^{-1} & & \frac{1}{2}\leq t<1.
\end{aligned}
\right.
$$

We end this section by generalizing the results to other random locations such as the location of the largest jump in a fixed interval.

In \cite{samorodnitsky:shen:2013} the authors introduced the notion of \textit{intrinsic location functional}, which is a large family of random locations including the location of the path supremum, the first hitting time to a fixed level, among many others. It was later shown in \cite{shen:2016} that there exists an equivalent characterization of the intrinsic location functionals using partially ordered random sets, which we take here as the definition.

Let $H$ be a space of real valued functions on $\mathbb R$, closed under translation. That is, for any $f\in H$ and $c\in\mathbb R$, $\theta_c f\in H$. Let $\mathcal I$ be the set of all compact, non-degenerate intervals in $\mathbb R$.\\

\begin{defn} [\cite{shen:2016}]
A mapping $L=L(f,I)$ from $H\times \mathcal I$ to ${\mathbb R}\cup \{\infty\}$ is called an intrinsic location functional, if
\begin{enumerate}
\item $L(\cdot, I)$ is measurable for $I\in\mathcal I$;
\item For each function $f\in H$, there exists a subset $S(f)$ of $\mathbb R$, equipped with a partial order $\preceq$, satisfying:
    \begin{enumerate}
    \item For any $c\in \mathbb R$, $S(f)=S(\theta_cf)+c$;
    \item For any $c\in\mathbb R$ and any $t_1, t_2\in S(f)$, $t_1\preceq t_2$ implies $t_1-c\preceq t_2-c$ in $S(\theta_cf)$,
    \end{enumerate}
    such that for any $I\in \mathcal I$, either $S(f)\cap I=\phi$, in which case $L(f,I)=\infty$, or $L(f,I)$ is the maximal element in $S(f)\cap I$ according to $\preceq$.\\
\end{enumerate}
\end{defn}

Briefly, an intrinsic location functional always takes the maximal element in a random set in the interval of interest, according to some partial order. Infinity was added as a possible value to deal with the case where some random location may not be well-defined for certain path and interval.

For the case of the location of the path supremum over an interval, the set $S(f)$ is the set of all the points $t\in\mathbb R$ such that $f(t)$ is the supremum of $f$ in either $[t-s,t]$ or $[t,t+s]$ (or both) for some $s>0$, and the order $\preceq$ is the natural order of the value $f(t)$. A review of the proofs in this section shows that they did not use any specific properties of the location of the path supremum, but rather two general properties that this location possesses, in terms of its partially ordered random set representation:

\begin{enumerate}
\item The set $S$ is a compatible set, as defined in Definition \ref{d:compatible}, with space $D'$ replaced by $H$;
\item The partial order $\preceq$ is also compatible with rescaling. That is, $t_1\preceq t_2$ in $S(f(\cdot))$ implies $c^{-1}t_1\preceq c^{-1}t_2$ in $S(d\cdot f(c~\cdot))$ for $c,d\in{\mathbb R}^+$.
\end{enumerate}

Consequently, all the results in Section \ref{s:general} can be generalized to any intrinsic location functional satisfying the two properties above. In particular, the spectral-type result, Theorem \ref{t:spectral}, and its two corollaries, also apply to the location of the largest jump in $[0,1]$, defined as
\begin{equation}\label{e:jump}
\delta_{\BX}:=\inf\left\{t\in [0,1]: |X(t)-X(t-)|=\sup_{s\in[0,1]}|X(s)-X(s-)|\right\},
\end{equation}
or the location of the largest drawdown, by considering only the downward jumps, among others.

\section{Supremum location of self-similar L\'{e}vy processes}

In this section we consider the special case of self-similar L\'{e}vy processes, which are ss,si processes with independent increments. Recall that in order to make the location of the path supremum well-defined, we are using the upper semicontinuous modification of the L\'{e}vy process.

\begin{proposition}\label{p:levy}
Let $\{X(t)\}_{t\in\mathbb R}$ be a self-similar L\'{e}vy process with exponent $H>0$ and satisfying Assumption U, and $\nu$ be the same as previously defined. Then
$$
\nu=\nu_1\times\nu_2,
$$
where $\nu_1$ and $\nu_2$ are measures on $(0,\infty)$, with survival functions $\overline{F}_1(l):=\nu_1(l,\infty)=l^{-c_1}$ and $\overline{F}_2(r):=\nu_2(r,\infty)=c_0r^{-c_2}$, respectively. The constants $c_0, c_1, c_2>0$ and $c_1+c_2=1$.
\end{proposition}

\begin{proof}
Consider $\nu((l,\infty)\times(r,\infty))$ for $l,r$ satisfying $l>1, r>1$. Notice that by the construction of the set of points $\mathcal E$, there is at most one point in $(1,\infty)\times(1,\infty)$. Thus
\begin{align*}
& \nu((l,\infty)\times (r,\infty))\\
=& P(\text{there exists } s\in[0,1], \text{ such that } l(s)> l, r(s)> r)\\
=& P(l(s)> l, r(s)> r|E)P(E),
\end{align*}
where the event $E:=\{\text{there exists a unique } s\in[0,1], \text{ such that } l(s)> 1, r(s)> 1\}$. By the independence of increments, we further have
\begin{align*}
 & P(l(s)> l, r(s)> r|E)P(E)\\
= & P(l(s)> l|l(s)> 1)P(r(s)> r|r(s)> 1)P(E)\\
=: & \overline{F}'_1(l)\overline{F}'_2(r)P(E), \quad l>1, r>1.
\end{align*}
The condition $l>1$ and $r>1$ is not essential due to the self-similarity. Thus
$$
\nu((l,\infty)\times (r,\infty))\propto \overline{F}_1(l)\overline{F}_2(r)
$$
for some functions $\overline{F}_1$ and $\overline{F}_2$. Taking $A=(l,\infty)\times(r,\infty)$ in (\ref{e:trans}), we have
$$
\overline{F}_1(al)\overline{F}_2(ar)=a^{-1}\overline{F}_1(l)\overline{F}_2(r)
$$
for any $a>0$.

Standard procedure leads to the conclusion that the only solutions of this functional equation, which make both $\overline{F}_1$ and $\overline{F}_2$ non-increasing, are of the form
$$
\overline{F}_1(l)=c'l^{-c_1},
$$
$$
\overline{F}_2(r)=c''r^{-c_2},
$$
where $c_1,c_2>0$ and $c_1+c_2=1$. Finally $c'$ and $c''$ can obviously be combined as $c_0$, and put only in front of $\overline{F}_2$.
\end{proof}

Proposition \ref{p:levy} leads to a new way to prove the following result regarding the distribution of $\tau_{\BX}$, the supremum location for the self-similar L\'{e}vy process $\BX$ over [0,1]. This result was first established in \cite{chaumont:2013} by considering the joint distribution of the location and the value of the path supremum for stable L\'{e}vy processes. The special case of Brownian motion is well known and can be found in, for instance, \cite{shepp:1979}. Also, note that for non-constant L\'{e}vy processes, Assumption U is automatically satisfied.

\begin{theorem}
Let $\mathbf{X}=\{X(t)\}_{t\in\mathbb R}$ be a self-similar L\'{e}vy process with exponent $H>0$, and $\tau_{\BX}$ the location of its path supremum over $[0,1]$. Then one of the three following scenarios is true:\\
1. $\BX$ is almost surely decreasing, hence $P(\tau_{\BX}=0)=1$;\\
2. $\BX$ is almost surely increasing, hence $P(\tau_{\BX}=1)=1$;\\
3. $\BX$ is not monotone, $\tau_{\BX}\sim Beta(1-c_1,1-c_2)$, where $c_1, c_2>0$ and $c_1+c_2=1$.
\end{theorem}

\begin{proof}
The first two cases are trivial. Now let us consider the case where $\BX$ is not monotone. It is clear by Theorem \ref{t:frame} and Proposition \ref{p:levy} that the density function of $\tau_{\BX}$ in $(0,1)$, $f(t)$, satisfies
$$
f(t)\propto t^{-c_1}(1-t)^{-c_2}, \quad 0<t<1.
$$
Therefore it suffices to prove that when $\BX$ is not monotone, $P(\tau_{\BX}=0)=P(\tau_{\BX}=1)=0$. Here we show that $P(\tau_{\BX}=0)=0$. Once this is done, $P(\tau_{\BX}=1)=0$ follows in a symmetric way.

Suppose $P(\tau_{\BX}=0)>0$. Self-similarity implies that
\begin{equation}\label{e:nopointmass}
P(\tau_{\BX,T}=0)=P(\tau_{\BX}=0)
\end{equation}
for any $T\geq 0$. Define
$$
\tau_{\BX,\infty}=\inf\{t\geq 0: X(t)=\sup_{s\geq 0} X(s)\}=\inf\{t\geq 0: X(t)=\sup_{s\geq t} X(s)\}
$$
to be the location of the path supremum of $\BX$ over $[0,\infty)$. In the case where the supremum is never achieved, define $\tau_{\BX,\infty}=\inf(\phi)=\infty$. Taking $T\to\infty$ in (\ref{e:nopointmass}), we have
$$
P(\tau_{\BX,\infty}=0)=P(\tau_{\BX}=0)>0.
$$

On the other hand, since $\tau_{\BX,\infty}\geq \tau_{\BX}$, and $\BX$ is not monotone, $P(\tau_{\BX,\infty}=0)\leq P(\tau_{\BX}=0)<1$. By the independence and stationarity of the increments, we have
\begin{align*}
&P(\tau_{\BX,\infty}\in(0,1])\\
=&P(\tau_{\BX}\in(0,1], \tau_{\BX}=\tau_{\BX,\infty})\\
\geq &P(\tau_{\BX}\in(0,1], X(1)=\sup_{s\geq 1}X(s))\\
=&P(\tau_{\BX}\in(0,1])P(X(1)=\sup_{s\geq 1}X(s))\\
=&P(\tau_{\BX}\in(0,1])P(\tau_{\BX,\infty}=0)\\
>&0.
\end{align*}
However, notice that the set $\{t\in {\mathbb R}: X(t)=\sup_{s\geq t} X(s)\}$ is a compatible set. Thus by Lemma \ref{l:compatible}, it is either dense in $\mathbb R$ or empty. As a result, $\tau_{\BX,\infty}=0$ or $\tau_{\BX,\infty}=\infty$ almost surely. This contradicts $P(\tau_{\BX,\infty}\in(0,1])>0$. Thus we conclude that the assumption can not be true, in other words, $P(\tau_{\BX}=0)=0$.
\end{proof}

\section*{Acknowledgements} The author would like to thank Xiaofei Shi for valuable input and discussions. The author acknowledges support from the Natural Sciences and Engineering Research Council of Canada (NSERC Discovery Grant number 469065).



\end{document}